\documentclass[reqno,11pt]{amsart}
\usepackage{amsmath, latexsym, amsfonts, amssymb, amsthm, amscd}
\usepackage[utf8]{inputenc}
\usepackage{graphics,epsf,psfrag,epsfig}

\usepackage[english,frenchb]{babel} 
\newenvironment{poliabstract}[1]
  {\renewcommand{\abstractname}{#1}\begin{abstract}}
  {\end{abstract}}

\usepackage{setspace}

\usepackage{anysize}
\marginsize{1in}{1in}{1in}{1in}

\numberwithin{equation}{section}

\newtheorem{theorem}{Theorem}[section]
\newtheorem{lemma}[theorem]{Lemma}
\newtheorem{proposition}[theorem]{Proposition}

\newtheorem{rem}[theorem]{Remark}

\renewcommand{\ge}{\geq}
\renewcommand{\le}{\leq}
\newcommand{\ind}{\mathbf{1}}

\newcommand{\R}{\mathbb{R}}

\newcommand{\N}{\mathbb{N}}
\renewcommand{\tilde}{\widetilde}
\renewcommand{\hat}{\widehat}

\newcommand{\cB}{{\ensuremath{\mathcal B}} }

\newcommand{\bP}{{\ensuremath{\mathbf P}} }
\newcommand{\bE}{{\ensuremath{\mathbf E}} }

\DeclareMathSymbol{\leqslant}{\mathalpha}{AMSa}{"36} % nicer `smaller or equal'
\DeclareMathSymbol{\geqslant}{\mathalpha}{AMSa}{"3E} % nicer `larger or equal'
\DeclareMathSymbol{\eset}{\mathalpha}{AMSb}{"3F}     % nicer `emptyset'
\renewcommand{\leq}{\;\leqslant\;}                   % redef. of < or =
\renewcommand{\geq}{\;\geqslant\;}                   % redef. of > or =
\newcommand{\dd}{\,\text{\rm d}}             % a straight d for differentials
\newcommand{\bbE}{{\ensuremath{\mathbb E}} }

\newcommand{\bbN}{{\ensuremath{\mathbb N}} }

\newcommand{\bbP}{{\ensuremath{\mathbb P}} }

\newcommand{\bbR}{{\ensuremath{\mathbb R}} }

\newcommand{\gb}{\beta}
\newcommand{\gga}{\gamma}            % \gg already exists...
\newcommand{\gd}{\delta}
\newcommand{\gep}{\varepsilon}       % \ge already exists...

\newcommand{\gD}{\Delta}

\newcommand{\go}{\omega}

\newcommand{\gl}{\lambda}

\makeatletter
\def\captionfont@{\footnotesize}
\def\captionheadfont@{\scshape}

\long\def\@makecaption#1#2{%
  \vspace{2mm}
  \setbox\@tempboxa\vbox{\color@setgroup
    \advance\hsize-6pc\noindent
    \captionfont@\captionheadfont@#1\@xp\@ifnotempty\@xp
        {\@cdr#2\@nil}{.\captionfont@\upshape\enspace#2}%
    \unskip\kern-6pc\par
    \global\setbox\@ne\lastbox\color@endgroup}%
  \ifhbox\@ne % the normal case
    \setbox\@ne\hbox{\unhbox\@ne\unskip\unskip\unpenalty\unkern}%
  \fi
  \ifdim\wd\@tempboxa=\z@ % this means caption will fit on one line
    \setbox\@ne\hbox to\columnwidth{\hss\kern-6pc\box\@ne\hss}%
  \else % tempboxa contained more than one line
    \setbox\@ne\vbox{\unvbox\@tempboxa\parskip\z@skip
        \noindent\unhbox\@ne\advance\hsize-6pc\par}%
\fi
  \ifnum\@tempcnta<64 % if the float IS a figure...
    \addvspace\abovecaptionskip
    \moveright 3pc\box\@ne
  \else % if the float IS NOT a figure...
    \moveright 3pc\box\@ne
    \nobreak
    \vskip\belowcaptionskip
  \fi
\relax
}
\makeatother
\def\writefig#1 #2 #3 {\rlap{\kern #1 truecm
\raise #2 truecm \hbox{#3}}}

\newcommand{\var}{{\rm Var}}

\begin{document}
 
 \title[Brownian Motion in Correlated Poissonian Potential]{Superdiffusivity for Brownian Motion in a Poissonian Potential with long range correlation I: \\
 Lower bound on the volume exponent}
\author{Hubert Lacoin}
\address{CEREMADE, Place du Mar\'echal De Lattre De Tassigny
75775 PARIS CEDEX 16 - FRANCE}
\email{lacoin@ceremade.dauphine.fr}
\maketitle

\selectlanguage{english}
\begin{poliabstract}{Abstract}
We study trajectories of $d$-dimensional Brownian Motion in Poissonian potential up to the hitting time of a distant hyper-plane. 
Our Poissonian potential $V$ is constructed from a field of traps whose centers location is given by a Poisson Point process and whose radii are IID distributed with a common distribution that has unbounded support; it has the particularity of having long-range correlation.
 We focus on the case where the law of the trap radii $\nu$ has power-law decay and prove that superdiffusivity hold under certain condition, and get a lower bound on 
 the volume exponent. Results differ quite much with the one that have been obtained for the model with traps of bounded radii by W\"uhtrich \cite{cf:W2, cf:W3}: the superdiffusivity phenomenon is enhanced by the presence of correlation.
\end{poliabstract}

\selectlanguage{frenchb}
\begin{poliabstract}{Résumé}
Dans cet article, nous étudions les trajectoires 
d'un mouvement brownien dans $\bbR^d$ évoluant dans un potentiel poissonien jusqu'au temps d'atteinte d'un hyper-plan situé loin  de l'origine. 
Le potentiel poissonien $V$ que nous considerons est construit à partir d'un champs de pièges dont les centres 
sont déterminés par un processus de Poisson et dont les rayons sont des variables aléatoires IID.
Nous concentrons notre étude sur le cas particulier ou la loi des rayons des pièges à une queue polynomiale et nous prouvons 
que les trajectoires ont un caractère surdiffusif quand certaines conditions sont vérifées et nous donnons une borne inférieure pour l'exposant de volume.
Les résultats sont sensiblement différents de ceux obtenus dans le cas ou les pièges sont à rayon bornés par   W\"uhtrich \cite{cf:W2, cf:W3}:
le phénomène de surdiffusivité est renforcé par la présence de corrélations.
\end{poliabstract}

\selectlanguage{english}

\section{Introduction}

\subsection{Brownian Motion and Poissonian Traps}

This paper studies a model of Brownian Motion in a random potential. Given a random function $V$ defined on $\bbR^d$ and $\gl, \gb > 0$ ($\gb$ being the inverse temperature)
we study trajectories of a Brownian motion $(B_t)_{t\ge 0}$  killed at (space-dependent) rate $\gb(\gl+V(B_t))$ conditioned to survive up to the hitting time of a distant hyperplane.
\medskip

The potential $V$ that we considered is buildt from a Poisson Point Process on $\bbR^d\times \bbR^+$ with intensity 
$\mathcal L\times \nu$ where $\mathcal L$ is the Lebesgue measure, 	and $\nu$ is a  probability measure. We call it $\go:=\{(\go_i,r_i) \ i\in \mathbb N\}$.
The definition of $V$ is  $V(\cdot):= \sum_{i\in \N} (r_i)^{-\gamma} W((r_i)^{-1}(\cdot -\go_i))$ where $W$ is a non-negative function 
with compact support (for the sake of simplicity
we restrict ourselves to $W=\ind_{B(0,1)}$, where $B(0,1)$ denotes the euclidian ball). The potential can be seen as a superposition of traps centered on the points $\go_i$,
and with IID random radii $r_i$. 
We are specifically interested in the case where $\nu$ has unbounded support and a tail with power law decay.

\medskip

This model is very similar to the ones studied in   \cite{cf:SCPAM, cf:SAOP, cf:W1, cf:W2, cf:W3} (see also the monograph of Sznitman \cite{cf:S} for a full acquaintance with the subject),
the only difference is that we allow the traps to have random radii.
 The crucial difference is that the potential we consider has long-range spacial correlation i.e.\ that the value of $V$ at two distant point are not independent but have correlation that decays like a power of the distance.
Another situation where one has a correlated potential $V$  is when $W$ is not compactly supported (see e.g. \cite{cf:DV,cf:Pa}) but this out of our scope.

\subsection{Superdiffusivity and volume exponent}

A typical trajectory of a Brownian Motion killed with homogeneous rate $\gl$ and conditioned to survive till it hits a distant hyperplane looks like the following:\\
The motion along the direction that is orthogonal to the hyperplane (call it $e_1$) is ballistic (with speed $1/\sqrt{2\gl}$) but  the motion along the $d-1$ other coordinate
is diffusive, and for that reason trajectories tend to stay in a tube centered on the axis $\bbR e_1$ of diameter $\sqrt{L}$ where $L$ is the distance between the motions starting point and  the hyperplane that has to be hit.

\medskip

Adding a non homogenous term to the killing rate makes the problem much harder to analyze and changes this behavior in some cases:
physicists predicts that when $\gb$ is large (at low temperature) or when $d< 4$ for every $\gb$, transversal fluctuation of the trajectories are superdiffusive\ i.e. 
of an amplitude $L^\xi$ for some $\xi\in (1/2,1)$ that is called \textit{the volume exponent}. The aim of the paper 
is to show that spatial correlation in $V$
enhances that phenomenon.

\section{Model and results}

\subsection{Model}

Let us make formal the definition we gave for the model.
We consider 
\begin{equation}\label{degome}
 \go:= \{ (\go_i,r_i) \ | i\in \N\} 
\end{equation}
 a Poisson Point Process in $\R^d \times \R^+$  (we index the points in the Poisson Point Process in an arbitrary deterministic way,  e.g. such that $|\go_i|$ is an increasing sequence, $|\cdot|$ being the euclidian norm on $\bbR^d$) whose intensity is given by $\mathcal L \times \nu$ where $\mathcal L$ 
is the Lebesgue measure on $\bbR^d$ and $\nu=\nu_{\alpha}$  is the probability measure on $\R^+$ defined by
\begin{equation}
\forall r\ge 1,\  \nu([r,\infty])=r^{-\alpha},
\end{equation}
for some $\alpha>0$ (which is a parameter of the model). We denote by $\bbP$, and $\bbE$ its  associated probability law and expectation. 
 
 \medskip

 Given $\gamma>0$, let $V^{\go}$,  $\bbR^d\to \bbR_+$ be defined as 
\begin{equation}
 V^{\go}(x):=\sum_{i=1}^{\infty} r_i^{-\gamma} \ind_{\{|x-\go_i|\le r_i\}}.
\end{equation}

Note that  $V^{\go}(x)<\infty$, for almost every realization of $\go$ and for every $x$ if and only if the condition
\begin{equation}
\alpha+\gga-d>0
\end{equation}
is fulfilled, and we always consider it to be so in the sequel.
This construction is natural way  to get a potential with long range correlation that decays like a power of the distance constructed from a Poisson Point Process.
Indeed with this setup,
\begin{equation}
 \bbE[V^{\go}(0)V^{\go}(x)]\asymp x^{d-\alpha-\gga}.
\end{equation}
\medskip

Given $x\in \bbR^d$, let $\bP_x$ (and $\bE_x$ the associated expectation) denote the law  $B=(B_t)_{t\ge 0}$, standard Brownian Motion
starting from $x$ and set $\bP:=\bP_0$.
Given $L>0$ set 
\begin{equation}
\mathcal H_L:= \{L\}\times \bbR^{d-1}
\end{equation}

\medskip

Given any closed set $A\subset \bbR^d$ let $T_A$ denote the hitting time of $A$.
Given $\gl>0,\gb>0$, 
the probability for a Brownian Motion killed with rate $\gb(V+\gl)$ to survive till it hits  $\mathcal H_L$ is equal to

\begin{equation}
 Z_{L}^{\go}:= \bE\left[\exp\left(-\int_{0}^{T_{\mathcal H_L}} \gb( V^{\go}(B_t)+\gl)\dd t\right)\right].
\end{equation}
The law of the trajectories conditioned to survival $\mu_L^{\go}$ is given by
\begin{equation}\label{defmu}
 \frac{\dd \mu_L^{\go}}{\dd \bP}(B):= 
\frac{1}{Z^\go_L}\exp\left(-\int_{0}^{T_{\mathcal H_L}} \gb(V^{\go}(B_t)+\gl)\dd t\right) .
\end{equation}
In what follows,  we consider only the case $\gb=1$ as temperature does not play any role in our results.

\medskip

%We call the points $(\go_i, r_i)$ \textit{traps}, $r_i$ is the radius of the trap, $\go_i$ its location and $r_i^{-\gamma}$ its intensity.
%The reason for doing so is the following interpretation of the model:\\
%Consider a Brownian Motion that is killed with rate $r_i^{-\gamma}$ when it lies in $B(\go_i,r_i)$ the Euclidian ball of center $\go_i$ and radius 
%$r_i$, the rule being that killing rates sums up if the Brownian Motion lies in several different traps.
%Then $Z_{L}^{\go}$ denotes the probability that the Motion has not been killed before reaching $\mathcal H_L$
%and $\mu_L^{\go}$ the distribution of trajectories conditioned to not being killed.

\subsection{Review of known results}

Let us turn to a rigorous definition of the volume exponent.
For $\xi>0$ one defines $\mathcal C_L^{\xi}$ the be a tube of cubic section, of width $L^\xi$ and centered on $\bbR e_1$, where $e_1=(1,0,\dots,0)$.

\begin{equation}\label{clxi}
\mathcal C_L^{\xi}:= \R \times [-L^{\xi}/2,L^{\xi}/2]^{d-1},
\end{equation}
and the event
\begin{equation}
\mathcal A_L^{\xi}:= \{ B \ | \ \forall  t \in [0, T_{\mathcal H_L}], A_t \in \mathcal C_L^{\xi}\}.
\end{equation}
In words, $A_L^\xi$ is the event:``$B$ has transversal fluctuation of amplitude less than $L^\xi$".

We define the volume exponent as

\begin{equation}\label{defxiz}
\xi_0:=\sup \{\xi>0 \ | \ \lim_{L\to\infty} \bbE\left[\mu_L(\mathcal A_L^\xi)\right]=0\}.
\end{equation}

It is expected to coincide with 

\begin{equation}
\xi_1:=\inf \{\xi>0 \ | \ \lim_{L\to\infty} \bbE\left[\mu_L(\mathcal A_L^\xi)\right]=1\}.
\end{equation}

In particular if $V\equiv 0$ one has $\xi_1=\xi_0=1/2$.

\medskip

Let us recall what are the conjecture and known result for the volume exponent for the model of Brownian Motion in Poissonian Obstacles studied in
\cite{cf:SCPAM, cf:SAOP, cf:S, cf:W1, cf:W2, cf:W3} and for related model. In the remainder of this section, $\xi_0$ relates more to the general notion of volume exponent
that to the strict definition given above and in the different results that are cited, definitions may differ:

When $d\ge 4$ whether $\xi_0>1/2$ or not should depend on the temperature i.e.\ on the value of $\gb$: at high temperature (low $\gb$) trajectories 
should be diffusive 
and satisfy invariance principle whereas at low temperature (high $\gb$) trajectories are believed to be superdiffusive. In the low temperature phase, the value of $\xi_0$ should not depend on $\gb$.
Diffusivity at high temperature has been proved for a discrete version of this model by Ioffe and Velenik  \cite{cf:IV} and it is reasonable to think that their technique can
adapt to the Brownian case when correlation have bounded range. Prior to that, similar results had been proved for directed polymer in random environment that can be considered as a simplified version of the model (see e.g.\ \cite{cf:B,cf:CY}). Superdiffusivity at low-temperature is a much more challenging issue:
physicists have no clear prediction for the value of $\xi_0$ and no mathematical progress towards proving $\xi_0>1/2$ has been made so far.

In any dimension, the value of $\xi$ is conjectured be related to the fluctuation of $\log Z^{\go}_L$ around its mean:
If the variance asymptotically satisfies
\begin{equation}
\var \log Z_L^{\go}\approx L^{2\chi},
\end{equation}
then one should have the scaling relation 
\begin{equation}
\chi=2\xi_0-1.
\end{equation}
The heuristic reason for this is that $L^{2\xi-1}$ is the entropic cost for moving $L^{\xi}$ away from the axis $\bbR e_1$, 
whereas the energetic gain one might expect for such a move is $L^\chi$.  The volume exponent corresponds to the value of $\xi$ for which cost and gain are balanced.
\medskip

When $d\le 3$, there is no phase transition and trajectories are expected to be superdiffusive for all $\gb$. It is not very clear what it means when $d=3$ but for
the two-dimensional case, physicists  predicts on heuristic ground that $\xi_0=2/3$ and $\chi=1/3$.
This is conjectured to hold not only for Brownian Motion in Poissonian Obstacles but for a whole family of two-dimensional models called the KPZ universality class
(directed last-passage percolation, first passage percolation, directed polymers...). In fact the conjecture goes much further and includes a description of the 
the scaling limit 
(see e.g.\ the seminal paper of Kardar Parisi and Zhang \cite{cf:KPZ}).

\medskip

A lot of efforts have been made to bring that conjecture on rigorous ground. In fact, it has even been proved that $\xi_0=2/3$ for some very specific models in the KPZ universality class:
\begin{itemize}
\item Directed last passage percolation in $1+1$ dimension with exponential environment by Johansson \cite{cf:J}.
\item Directed polymer in $1+1$ dimension with $\log$-Gamma environment and specific boundary condition by Seppalainen  \cite{cf:S}.
\end{itemize}
These two results have in common that they have been proved by using exact computation that are specific to the model.
Note that a similar result has been proved for the conjectured scaling limit of this model, in \cite{cf:BQS}.

\medskip

Another approach has been to look for more robust method using the idea of energy vs. entropy competition.
In \cite{cf:W2,cf:W3} , W\"uhtrich  proved that $\xi_0\ge 3/5$ for $d=2$ and that $\xi_0\le 3/4$ in all dimension
 (with a definition for $\xi_0$ that is slightly differs of the one we present here).  In \cite{cf:W1}, he proved a rigorous version of the scaling identity $\chi=2\xi_0-1$.
Similar results had been proved before for first passage Percolation by Licea, Newman and Piza \cite{cf:LNP}  and after for directed polymers by Peterman
\cite{cf:P} and M\'ejane \cite{cf:M}.

%When $V\ne 0$ this picture is expected to change, motion along the first coordinate is still ballistic, but (at least in the in low-dimensional case),
%transversal fluctuation are expected to be much larger order, i.e. to behave like $L^\xi$ for some $\xi>1/2$ called \textit{volume exponent}.
%The reason for thinking so is that when $V$ is in-homogenous, $B$ may go further away than $\sqrt{L}$ from the axe to reach regions where 
%the potential is more favorable (i.e. lower).

\medskip

In \cite{cf:Lac}, we have investigated the effect of transversal correlation in the environment for directed polymers, and in particular their effect on the volume exponent.
There it is shown that in any dimension, if environment correlations decay like a small power of the distance then, superdiffusivity holds.
More precisely that if the correlation decays like the inverse-distance to the power $\theta$, then $\xi_0\ge 3/(4+\theta)$.
In some cases it shows in particular that $\xi_0>2/3$ which indicates that KPZ conjecture does not holds in that case.
The bound $\xi\le 3/4$ of \cite{cf:M} remains valid.

\medskip

Here we study the effect of isotropic correlation (and therefore it seemed natural to to it in for an undirected model), and we have not found in the literature 
any prediction about what the value of $\xi_0$ should be.

%
%\subsection{Review of previous results on volume exponent}
%
%This behavior is supposed to hold for a large class of model, but has been proved rigorously only in a very few particular cases.
%In some special versions of related model (first passage directed percolation and directed polymers), the exact equality $\xi=2/3$ for $d=2$ has been obtained
%\cite{cf:J}, \cite{cf:Sepa}. In both cases the result relies on exact calculation and yet cannot be adapted to more general cases.
%In \cite{cf:BQS} a similar result is derived for the conjectured scaling-limit of this family of model.
%
%\medskip
%
%On the other hand, using the general idea of entropy vs. energy competition, bounds have been obtained for Brownian Motion in Poissonian obstacles \cite{cf:W1,cf:W2,cf:W3} and some model in the same universality class, directed polymers in random environment \cite{cf:P, cf:M}, first-passage percolation \cite{cf:LNP}.
%For all of these model it has been proved that $\xi\ge 3/5$ for $d=2$ and $\xi\le 3/4$ in every dimension.
%In \cite{cf:Lac} the effect of transversal correlation in the disorder is studied for directed polymer, and it is shown that correlation tend to enhance the superdiffusivity properties:
%if correlations are sufficiently slowly decaying $\xi$ can be arbitrarily close to $3/4$ in any dimension. 
%
%\medskip
%
%The aim of the study of this model is to investigate the effect of isotropic correlation in an undirected model. 

\subsection{Main Result}

We present a lower bound bound on $\xi_0$ for our model with correlation.

Set 
\begin{equation}\begin{split}
 \bar \xi(d,\alpha,\gamma)&:= \frac{1}{\alpha-d+1} \quad \text{if $\gamma\le \alpha-d$}\\
 \bar \xi (d,\alpha,\gamma)&:= \frac{3}{3+\alpha+2\gamma-d} \quad \text{if $\gamma\ge \alpha-d$}
\end{split}
\end{equation}

%For any $\xi>0$ define
%\begin{equation}
%\mathcal C_L^{\xi}:= \R \times [-L^{\xi}/2,L^{\xi}/2]^d.
%\end{equation}
%and the event
%\begin{equation}
% \mathcal A_L^{\xi}:= \{ B \ | \ \forall t \in [0, T_{\mathcal H_L}], B_t\in \mathcal C_L^{\xi}\},
%\end{equation}
%the event hat the trajectory stays confined in a tube of diameter $L^{\xi}$.

\begin{theorem}\label{superdiff} (Lower bound for the volume exponent)

For any choice of $\alpha$, $d$, $\gamma$, one has

\begin{equation}
\xi_0\ge \bar \xi(d,\alpha,\gamma)\vee (1/2)
\end{equation}

where $\xi_0$ is the quantity defined in \eqref{defxiz}.
\end{theorem}

\begin{rem}\rm
In some cases, the lower bound that we get for $\xi_0$ is larger than $3/4$, which contrasts with all the results that we have reviewed in the previous section and indicates that
isotropic correlation enhance superdiffusivity in a more drastic way than transversal ones.
The above result gives a necessary condition for having superdiffusivity: $\gamma<\alpha-d$ and $\alpha-d<1$ or $\gamma>\alpha-d$ and $\alpha+2\gamma-d<3$. 
\end{rem}

\begin{rem}\rm
The definition of the volume exponent that we use is different of the one used in \cite{cf:W2} which is slightly weaker.
Combining  techniques used in \cite{cf:Lac} and here one could prove also that $\xi_0\ge 3/5$ when $d=2$ for any value of $\alpha$ and $\gamma$
for this definition of $\xi_0$.
\end{rem}

\subsection{Further questions}

We prove in this paper that for a class of correlated environment, the trajectories have superdiffusive behavior
and that the bound $\xi\le 3/4$ that is valid for the uncorrelated model \cite{cf:W2} is not valid here and can be beaten.
Therefore one would be interested to find an upper bound ($<1$) for $\xi$. We have addressed this issue in companion paper \cite{cf:L2}. In some 
special cases (when either $\gamma=\alpha-d>1/3$) one can even prove that the lower bound that we prove here is optimal and give the exact value of $\xi_0=\xi_1$.

\medskip

The result that we present concerns the so-called point-to-plane model. A similar result should hold for the point-to-point model.
The method that we use in Section \ref{addt} and \ref{tunapa} are quite robust and could be easily adapted to the other setup but getting something 
similar to what is done in Section \ref{premsr} seems more difficult and challenging and we are not able to do it yet. One can still get a non optimal
result by using another construction inspired by what is done in \cite{cf:W3}, we present it in the Appendix.

\medskip

For the Brownian directed polymer in correlated environment, in \cite{cf:Lac},  it is shown that either superdiffusivity holds  at all temperature
or that one has diffusivity at high temperature (except for some special limiting cases) . For the model presented here one would like to show something similar e.g.\ that
diffusivity holds if correlation have fast-decay at infinity (decay like  a large power of the inverse-distance) and the amplitude of $V$ is small. 
For the moment this is quite out of reach and the methods used in \cite{cf:IV} do not seem to adapt to this case.

\section{Proof of Theorem \ref{superdiff}}

%\begin{equation}
%\mathcal C_L^{\xi}:=[0,L]\times [-L^{\xi}/2,L^{\xi}/2]^d.
%\end{equation}
%and the event
%\begin{equation}
% \mathcal A_L^{\xi}:= \{ B \ | \ \forall t \in [0, T_{\mathcal H_L}], B_t\in \mathcal C_L^{\xi}\},
%\end{equation}

%
%What one would like to show is that when $\xi\le f(\alpha,\gga)$.
%\begin{equation}
%\lim_{L\to \infty} \mu_L^{\go}( \mathcal A_L^{\xi} )=0 
%\end{equation}
%in $\bbP$ probability.

\subsection{Sketch of proof}

In order to make the strategy of the proof clear we 
need to introduce some notation. One defines 
\begin{equation}\label{bariri}
\bar C_L^\xi := [L/2,L] \times [-L^{\xi}/2,L^{\xi}/2]^{d-1}
\end{equation}
and 
\begin{equation}\label{tildidi}
\tilde C_L^\xi:= \{ x\in \bbR^d \ | \ d(x, \bar C_L)\le 2 \sqrt{d} L^{\xi} \}=  \bigcup_{y\in \bar C_L} B(y,2 \sqrt{d} L^{\xi}),
\end{equation}
where for a closed set $A\subset \bbR^d$,  and $x\in \bbR^d$, $d(x,A)$ denotes the Euclidean distance between
$x$ and $A$, i.e.
\begin{equation}
d(x,A):=\min_{y\in A} |y-x|
\end{equation}
($|\cdot|$ is the Euclidean norm), and $B(x,r)$, $r\ge 0$ is the Euclidean ball of radius $r$.
Let $\mathcal B^{\xi}_L$ be the set of trajectories that avoids the set $\tilde C_L^\xi$.
\begin{equation}
\mathcal B^{\xi}_L:=  \{ B \ | \ \forall t \in [0, T_{\mathcal H_L}], B_t\notin \tilde C_L^\xi \}.
\end{equation}
Note that, as Brownian trajectories are continuous
\begin{equation}
\mathcal B^{\xi}_L\cap \mathcal A^{\xi}_L=\emptyset.
\end{equation}
The first step of our proof (Section \ref{premsr})  is inspired by \cite{cf:LNP}. We prove a result much weaker that Theorem \ref{superdiff} by 
using a simple geometric argument combined to rotational invariance: that with probability close to one,

\begin{equation}
\mu^{\go}_L(\mathcal A_L^{\xi})\le  e^{L^{2\xi-1}(\log L)^3}\mu^{\go}_L(\mathcal B^{\xi}_L),
\end{equation}
or equivalently, that with probability close to one,

\begin{equation}\label{machinbidul}
Z^{\go}_L(\mathcal A_L^{\xi})\le e^{L^{2\xi-1}(\log L)^3}Z^{\go}_L(\mathcal B^{\xi}_L).
\end{equation}
where for an event $A$, we use the notation
\begin{equation}
 Z_L^{\go}(A):=  Z_L^{\go}\times \mu_L^{\go}(A)= \bE\left[\exp\left(\int_{0}^{T_{\mathcal H_L}} (V^{\go}(B_t)+\gl)\dd t\right)\ind_{A}\right].
\end{equation}

Then we modify slightly the environment ($\go\to \tilde \go$) by adding additional traps whose radii are in $(\sqrt{d}L^\xi,2\sqrt{d}L^\xi)$, and whose 
centers are in the region $ \bar C^{\xi}_L$.
The second step of the proof (Section \ref{addt}) is to show that typical realization of  $\tilde \go$ are roughly the same as typical realization of $\go$.

\medskip

Finally , we notice that adding these traps lowers the value of $Z^{\go}_L(\mathcal A_L^{\xi})$ but  that  $Z^{\go}_L(\mathcal B^{\xi}_L)=Z^{\tilde\go}_L(\mathcal B^{\xi}_L)$.
(adding these traps changes the values taken by $V$ only in the  region $\tilde C_L^\xi$ that the trajectory in the event $\mathcal B^{\xi}_L$ do not visit). The third step of the proof (Section \ref{tunapa}) is to show that
with our choice of $\tilde \go$ and $\xi$, one has with large probability

\begin{equation}
Z^{\tilde \go}_L(\mathcal A_L^{\xi})\le e^{-L^{2\xi-1+\gep}} Z^{ \go}_L(\mathcal A_L^{\xi}).
\end{equation}
for some $\gep>0$, which combined with \eqref{machinbidul}, gives the result with $\tilde \go$ instead of $\go$.
The fact that $\go$ and $\tilde \go$ look typically the same allows to conclude.

\medskip

We explain in the course of the proof the reasons 
for our choices of $\tilde \go$ and how we obtain the condition on $\xi$.

\subsection{Using rotational invariance}\label{premsr}

For $\theta\in (-\pi/2,\pi/2)$, let $R_{\theta}$ denote following the rotation of $\R^d$
\begin{equation}
 (x_1,x_2,x_3,\dots,x_d)\mapsto (x_1\cos \theta- x_2\sin \theta, x_2 \cos \theta+x_1 \sin \theta, x_3, \dots, x_d).   
\end{equation}
Set
\begin{equation}
 \mathcal H^{\theta}_L:=R_\theta[\mathcal H_L]  \quad \text{and} \quad \mathcal C_L^{\theta,\xi}=R_{\theta}\mathcal C^\xi_L.
\end{equation}
(the image of the sets $\mathcal H_L$ resp.\ $\mathcal C^{\theta,\xi}_L$ for $R_{\theta}$).
One defines in the same fashion the event $\mathcal A_L^{\theta,\xi}$  as
\begin{equation}
\mathcal A_L^{\theta,\xi}:= \{ B\ | \ \forall t \in [0, T_{\mathcal H^{\theta}_L}], B_t\in \mathcal C_L^{\theta,\xi} \}.
\end{equation}
Note that if $B \in \mathcal A_L^{\theta, \xi}$ if and only if $R_{-\theta} B\in \mathcal A_L^{\xi}$.

One proves the following 

\begin{proposition}\label{preswe}
For any $\xi \in (0,1)$ set $\theta=\theta(L,\xi):=10 \sqrt{d} L^{\xi-1}$. Then one has that for any $N\le \gd \theta^{-1}$ (for some fixed small enough  $\gd >0$),
\begin{equation}
\bbP\left[ Z^{\go}_L(\mathcal A_L^{\xi})\ge 
e^{2 N^2 \theta^2 L\sqrt{\log L}}\!\!\!\!\!\! \max_{i\in \{-N,\dots,N \}\setminus\{0\}} \!\!\!\!\!\! Z^{\go}_L(\mathcal A_L^{i\theta,\xi}\cap \mathcal B^{\xi}_L )\right] \le  \frac{1}{L}+\frac{1}{N}.
\end{equation}
In particular, setting $N:= \log L$ one has 
\begin{equation}
\bbP\left[ Z^{\go}_L(\mathcal A_L^{\theta,\xi})\ge e^{(\log L)^3 L^{2\xi-1}} Z^{\go}_L(\mathcal B^{\xi}_L )\right] \le   \frac{2}{\log L}.
\end{equation}
\end{proposition}

We split the proof of the Proposition into two lemmas:
The first lemma allows to compare almost deterministically
$Z^{\go}_L(\mathcal A_L^{\theta,\xi})$ with $ Z^{R_{-\theta}(\go)}_L(\mathcal A_L^{\xi})$
(which by rotation invariance of $\go$ is distributed like $Z^{\go}_L(\mathcal A_L^{\xi})$).
$R_\theta(\go)$ denotes the image of the Poisson Point Process $\go$ by $R_{\theta}$, i.e.\ (recall \eqref{degome}) 

\begin{equation}
R_{\theta}(\go):= \{ (R_{\theta}\go_i,r_i) \ | i\in \N\} .
\end{equation}

\begin{lemma} \label{milor}
Set $\xi\in(0,1)$, $\theta$ such that $|\theta|\ge 10 \sqrt{d} L^{\xi-1}$ and $|\theta|\le \gd$ for some $\gd>0$, and $\go$ that satisfies $\max_{x\in[-L^2,L^2]^d} V^{\go}(x)\le \log L $.
Then for all sufficiently large $L$ one has
\begin{equation}
Z^{\go}_L(\mathcal A_L^{\theta,\xi}\cap \mathcal B^{\xi}_L)> Z^{R_{-\theta}(\go)}_L(\mathcal A_L^{\xi})\exp(-2 \theta^2 L\sqrt{\log L}) 
\end{equation}
\end{lemma}

The second lemma estimates the probability that $Z^{\go}_L(\mathcal A_L^{\xi})$ has the largest value among the different 
$(Z^{R_{i\theta}(\go)}_L(\mathcal A_L^{\xi}))_{i\in \{-N,\dots,N\}}$. The argument comes from \cite{cf:P},

\begin{lemma}\label{rotatruc}
For any value of $\theta$ and any $N$
\begin{equation}
\bbP\left[ Z^{\go}_L(\mathcal A_L^{\xi})> \max_{i\in\{-N,\dots,N\}\setminus\{0\}} Z^{R_{i\theta}(\go)}_L(\mathcal A_L^{\xi})\right]\le \frac{1}{N}
\end{equation}

\end{lemma}

The proof for of the Proposition from the lemmas is straightforward with the use of Lemma \ref{troto} that ensures that with probability $1/N$ the assumption on $V$ in Lemma \ref{milor} is satisfied.

\begin{proof}[Proof of Lemma \ref{milor}]
%Define
%\begin{equation}
%\mathcal C''_L(\xi):= \{x\in \R^d, d(x, \mathcal C'_L(\xi)) \le 2\sqrt{d}L^{\xi}\}.
%\end{equation}
By symmetry we can assume $\theta>0$.
 The assumptions  we have on $\theta$ guarantees that on the event $\mathcal A_L^{\theta,\xi}$, 
$T_{\mathcal H^{\theta}_L}<T_{\mathcal H_L}$, and that 
$\mathcal C_L^{\theta,\xi} \cap \tilde C_L^\xi=\emptyset$.
(see figure \ref{yty}). Therefore using the strong Markov property for Brownian Motion,
\begin{multline}\label{ztz}
 Z^{\go}_L(\mathcal A_L^{\theta,\xi}\cap \mathcal B^{\xi}_L)= \bE\left[\exp\left(-\int_{0}^{T_{\mathcal H^{\theta}_L}} (V^{\go}(B_t)+\gl)
\dd t\right) \right. \\ \left.
\ind_{\mathcal A_L^{\theta,\xi}}\bE_{B_{ T_{\mathcal H^{\theta}_L}}} \left[
\exp\left(-\int_{0}^{T_{\mathcal H_L}} (V^{\go}(B_t)+\gl)\dd t\right)\ind_{\{\forall s\le T_{\mathcal H_L}, B_s \notin\tilde C^\xi_L\}}\right]\right].
\end{multline}

\begin{figure}[hlt]
\begin{center}
\leavevmode %\epsfysize =5 cm
\epsfxsize =14 cm
\psfragscanon
\psfrag{L2}{$L/2$}
\psfrag{L}{$L$}
\psfrag{R2}{$\R^{d-1}$}
\psfrag{R}{$\R$}
\psfrag{HL}{$\mathcal H_L$}
\psfrag{HLTHETA}{$\mathcal H_L^{\theta}$}
\psfrag{theta2xi}{$O(\theta^2L)$}
\psfrag{Lxi}{$L^\xi$}
\psfrag{O}{$0$}
\epsfbox{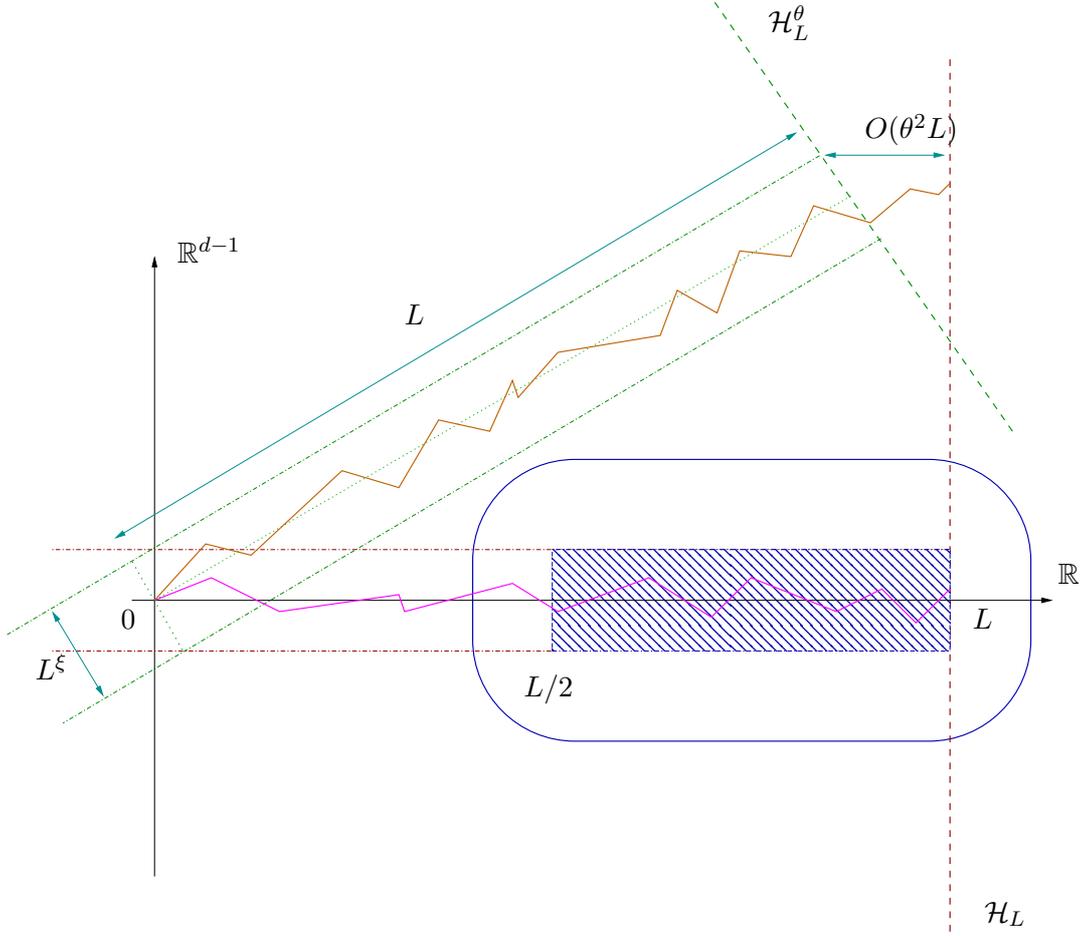}
\end{center}
\caption{\label{yty} Projection of the model along the two first coordinates. The two tubes represented are $\mathcal C_L^\xi$
and $\mathcal C_L^{\xi,\theta}$. The shadowed region is $\bar C_L^{\xi}$, and this is where $\go$ is modified.
The full line that encircles $\bar C_L^{\xi}$ denote the limit of $\tilde C_L^\xi$, the region where we may have $V^{\tilde \go}\ne V^\go$.
The two trajectories represent typical trajectories of $\mathcal A_L^{\xi}$ and $\mathcal A_L^{\xi,\theta}\cap \mathcal B_L^\xi$.
One can see on the picture that if $\theta$ is chosen large enough ($\theta\ge C L^{\xi-1}$ where $C$ is a constant depending on $d$),
one the event $\mathcal A^{\theta,\xi}_L$, the hitting time of $\mathcal H_L$ is larger than the hitting time of $\mathcal H_L^{\theta}$.
Moreover the tube $\mathcal C_L^{\xi,\theta}$ and the set $\tilde C_L^\xi$ are disjoint.
Standard trigonometry allows to say that the maximal distance between a point in $\mathcal C_L^\xi\cap \mathcal H_L^{\theta}$ and $\mathcal H_L$
is $O(\theta^2L)$. }
\end{figure}
On the event $\mathcal A_L^{\theta,\xi}$,  one has $B_{ T_{\mathcal H^{\theta}_L}}\in(\mathcal H^{\theta}_L\cap \mathcal C_L^{\theta,\xi} )$.
Therefore, the right-hand side of \eqref{ztz} is smaller than
\begin{multline}
\bE\left[\exp\left(-\int_{0}^{T_{\mathcal H^{\theta}_L}} (V^{\go}(B_t)+\gl)
\dd t\right) \ind_{\mathcal A_L^{\theta,\xi}}\right]\\
\times
\max_{x\in \mathcal H^{\theta}_L\cap \mathcal C_L^{\theta,\xi}} \bE_{x} \left[
\exp\left(-\int_{0}^{T_{\mathcal H_L}} (V^{\go}(B_t)+\gl)\dd t\right)\ind_{\{\forall s\le T_{\mathcal H_L}, B_s \notin \tilde C_L^\xi\}}\right].
\end{multline}
The first term on the above product is equal to $Z^{R_{-\theta}(\go)}_L(\mathcal A_L^{\xi})$.
By the assumption one has on $V$ ($V\le \log L$ in the ball of radius $L^2$)
, the second term is larger than 
\begin{equation}
 \max_{x\in \mathcal H^{\theta}_L\cap \mathcal C_L^{\theta,\xi}} \bE_x \left[
e^{-T_{\mathcal H_L} (\log L+\gl)}\ind_{\{ \forall s\le T_{\mathcal H_L}, B_s \notin \tilde C_L^\xi, |B_s|\le L^2 \}}\right].
\end{equation}
Hence the  Lemma is proved if one can show that for all $x$ in $ \mathcal H^{\theta}_L\cap \mathcal C_L^{\theta,\xi}$
\begin{equation}
 \bE_x \left[
e^{-T_{\mathcal H_L} (\log L+\gl)}\ind_{\{\forall s\le T_{\mathcal H_L}, B_s \notin \tilde C_L^\xi, |B_s|\le L^2\}}\right] \ge
\exp(-2\sqrt{\log L}\theta^2 L).
\end{equation}
From our assumptions on $\theta$, for $L$ large enough one has

\begin{equation}\label{aa}
 \max_{x\in \mathcal H^{\theta}_L\cap \mathcal C_L^{\theta,\xi}} d(x, \mathcal H_L)=(L^{\xi}/2)\sin\theta+L(1-\cos\theta)\le L\theta^2,
\end{equation}
and

\begin{equation}\label{aaa}
  \min_{x\in \mathcal H^{\theta}_L\cap \mathcal C_L^{\theta,\xi}} d(x, \tilde C_L^\xi)\ge \sin \theta L - (1+\sqrt{d})L^{\xi}\ge L\theta^2,
\end{equation}
(these inequality comes from the assumption one has taken for $\theta$ and trigonometry).

If one consider a $d$-dimensional cube whose edges are parallel to the coordinate axis centered at $x$ and of side-length $d(x, \mathcal H_L)$, then with $\bP$ probability $1/2d$, the exit time of the cube for a Brownian Motion started from $x$ is equal to $T_{\mathcal H_L}$. Moreover, if $L$ is large enough then this cube does not intersect 
$\tilde C^\xi_L$ (cf.\ \eqref{aa} and \eqref{aaa})
and lies within the ball of radius $L^2$ .
Hence (using symmetries of the cube)

\begin{equation}
 \bE_x \left[
e^{-T_{\mathcal H_L} (\log L+\gl)}\ind_{\{\forall s\le T_{\mathcal H_L}, B_s \notin \tilde C_L^\xi, |B_s|\le L^2\}}\right]\ge 
\frac{1}{2d} \bE \left[e^{-\mathcal T_{d(x, \mathcal H_L)} (\log L+\gl)}\right],
\end{equation}
where
\begin{equation}
 \mathcal T_{r}:= \inf\{t\ge 0, \|B_t\|_{\infty}=r\},
\end{equation}
and $\|x\|_{\infty}=\max_{i=1..d} |x_i|$ is the $l_{\infty}$ norm on $\bbR^d$.
The hitting time  $\mathcal T_{r}$ is stochastically dominated by $\tau_r$ the first hitting time of $r$ by a one dimensional Brownian motion.
And one has
\begin{equation}
 \bP(\tau_r\le s)=2\int_{|x|>(r/\sqrt{s})}\frac{1}{\sqrt{2\pi}}e^{-x^2/2}\dd x.
\end{equation}
Hence one has for $L$ large enough
\begin{multline}
 \bE_x \left[
e^{-T_{\mathcal H_L} (\log L+\gl)}\ind_{\{\forall s\le T_{\mathcal H_L}, B_s \notin \tilde C_L^\xi \text{ and } |B_s|\le L^2\} }\right]
\\
\ge \frac{1}{2d} \bE\left[e^{-\tau_{L\theta^2}(\log L+\gl)}\right]> e^{-2 \theta^2 L\sqrt{\log L}}.
\end{multline}

\end{proof}
% 
% By standard tubular estimates one has 
% 
% From our  assumptions
% \begin{multline}
%  \bE_{B_{T_{\mathcal H^{\theta}_L}}} \left[\exp(\left(-\int_{0}^{T_{\mathcal H_L}} (V^{\go}(B_t)+\gl)\dd t\right)\right]\\
% \ge \min_{x\in \mathcal H^{\theta}_L\cap B(x_{\theta,L},L^{\xi})} \bE_x
% \left[e^{-(\gl+\log L) T_{\mathcal H_L}} \ind_{\max_{t<\mathcal H_L} |B_t|\le L^2}\right].
% \end{multline}
%  Now 
% \begin{equation}
% \max_{x\in \mathcal H^{\theta}_L\cap B(x_{\theta,L},L^{\xi})} d(x, \mathcal H_L)=L^{\xi} \sin\theta+L(1-\cos\theta)\le L\theta^2
% \end{equation}
% so that for any value of $T$
% \begin{multline}
%   \min_{x\in \mathcal H^{\theta}_L\cap B(x_{\theta,L},L^{\xi}} \left[e^{-(\gl+\log L) T_{\mathcal H_L}} \ind_{\max_{t<\mathcal H_L} |B_t|
% \le L^2}\right]\\
% \ge e^{-(\gl+\log L) T} P(b_T\ge L\theta^2)-\bP \left[\ind_{\max_{t<T} |B_t|\le L^2-L}\right].
% \end{multline}
% One may choose $T=\theta^2 L(2\log L)^{-1/2}$ and get overall
% \begin{equation}
%  \bE_{B_{ T_{\mathcal H^{\theta}_L}}} \left[e^{-\gl T_{\mathcal H^{\theta}_L}} e^{-\int_{0}^{T_{\mathcal H_L}} V^{\go}(B_t)\dd t}\right]
% \ge e^{-4 \sqrt{\log L} \theta^2 L}.
% \end{equation}
% Therefore
% \begin{equation}
%   Z^{\go}_L(\mathcal C_L^{\theta,\xi})\ge e^{-4 \sqrt{\log L} \theta^2 L} \bE\left[e^{-\gl T_{\mathcal H^{\theta}_L}}e^{\int_{0}^{T_{\mathcal H^{\theta}_L}} V^{\go}(B_t)\dd t}
% \ind_{\mathcal C_L^{\theta}}\bE_{B_{ T_{\mathcal H^{\theta}_L}}}\right]
% =e^{-4 \sqrt{\log L} \theta^2 L} Z^{R_{-\theta}(\go)}_L(\mathcal C_L^{\xi}).
% \end{equation}
% \end{proof}

\begin{proof}[Proof of Lemma \ref{rotatruc}]

Note that $Z^{R_{i\theta}(\go)}_L(\mathcal A_L^{\xi})$, $i\in \{-N,\dots,N\}$ are identically distributed variables. 
However they are not exchangeable, and therefore the statement is not that obvious.
As
\begin{equation}
\sum_{k\in\{1,\dots,N\}} \bbP\left[Z^{R_{k\theta}\go}_L(\mathcal A_L^{\xi})> \max_{i\in\{1,\dots,N\}\setminus\{k\}}Z^{R_{i\theta}\go}_L(\mathcal A_L^{\xi})\right]\le 1
\end{equation} 
there exists some  $k_0\in \{1,N\}$ such that 
\begin{equation}
\bbP\left[Z^{R_{k_0\theta}\go}_L(\mathcal A_L^{\xi})> \max_{i\in\{1,\dots,N\}\setminus\{k_0\}} 
Z^{R_{i\theta}(\go)}_L(\mathcal A_L^{\xi})\right]\le \frac{1}{N}.
\end{equation}

Hence by rotational invariance of $\go$ and $V(\go)$ 
\begin{equation}
 \bbP\left[Z^{\go}_L(\mathcal A_L^{\xi}) > \max_{i\in\{1-k,\dots,N-k\}\setminus\{0\}} Z^{R_{i\theta}(\go)}_L(\mathcal A_L^{\xi})\right]\le \frac{1}{N}.
\end{equation}
\end{proof}

\subsection{Change of environment: Adding traps in $\bar C^{\xi}_L$}\label{addt}

With $\go$ we construct a second environment $\tilde \go$ that has more traps with  radius $\approx L^\xi$ in the region $\bar C_L^\xi$.
The aim of this section is to show that typical event for $\go$ are also typical for $\tilde \go$.

\medskip

We construct $\go$ and $\tilde \go$ on the same probability space and for convenience denote by $\bbP$ their joint probability.
Recall that  $\go$ is Poisson Point Process in $\R^d \times \bbR^+$ with intensity $\mathcal L \times \nu$.
Then define $\hat \go$ to be a Poisson Point Process on $\R^d \times\bbR^+$ independent of $\go$  with intensity
\begin{equation}
L^{-\frac{(d+1+\alpha)\xi+1}{2}} \tilde {\mathcal L},
\end{equation}
where $\tilde {\mathcal L}= \tilde {\mathcal L}(\xi,L)$ denotes the Lebesgue measure on the set
\begin{equation}
           \bar C_L ^\xi       \times [\sqrt{d} L^{\xi}, 2\sqrt{d} L^{\xi}].
\end{equation}
% where $\mathcal C'_L$ is the second half of the cylinder $\mathcal C_L$.
% \begin{equation}
%    \mathcal C'_L := [L/2,L] \times [-L^{\xi}/2,L^{\xi}/2]^{d-1}.
% \end{equation}
and define
\begin{equation}
 \tilde \go= \go+ \hat \go,
\end{equation}
which is a Poisson Point Process on $\R^d \times \bbR^+$ with intensity 
\begin{equation}
  \mathcal L \times \nu + L^{-\frac{(d+1+\alpha)\xi+1}{2}} \tilde {\mathcal L}.
\end{equation}

% The change between $\bbP$
% 
%  and $\tilde \bbP$ is that the intensity of trap whose radius lies in 
% $[2L^{\xi}, 4L^{\xi}$ has been intensified in the second half of the tube  $\mathcal C'_L$.
% 
% One has the following density.
% \begin{equation}
% \frac{\dd \tilde \bbP}{\dd \bbP}(\go)= \prod_{\{i, (\go_i,r_i)\in \mathcal C'_L\}} \left(1+(\alpha)^{-1} r_i^{1+\alpha}L^{-\frac{(d+1+\alpha)\xi+1}{2}}\right)
% \exp \left(-\frac{1}{2}   L^{\frac{d+1-\alpha)\xi+1}{2}}\right).
% \end{equation}
% 
% 

\begin{lemma}\label{mezchange}
Assume that
\begin{equation}
 \xi(d-1-\alpha)+1>0.
\end{equation}
Then, there exists a constant $C$ not depending on $L$ such that
for any event $A$ one has

\begin{equation}
 \bbP(\go \in A)\le C \sqrt{\bbP(\tilde \go \in A)}.
\end{equation}
 \end{lemma}

Before going to the proof, we explain why the result holds:
For $\go$ the number of points in $\tilde C^\xi_L\times [\sqrt{d} L^{\xi}, 2\sqrt{d} L^{\xi}]$ is 
a Poisson variable of mean
\begin{equation}
 \frac{(1-2^{-\alpha})}{2}d^{-\alpha/2} L^{(d-1-\alpha)\xi+1}.
\end{equation}
The fluctuation around the mean are therefore of order $L^{\frac{(d-1-\alpha)\xi+1}{2}}$.
The number of points in process $\hat \go$ is a Poisson variable of mean $L^{-\frac{(d+1+\alpha)\xi+1}{2}}\times \frac{\sqrt{d} L^{d\xi+1}}{2}$ (intensity $\times$ volume).
Therefore the number of points one add to $\go$ to get $\tilde \go$ is of the same order as 
the fluctuation for the number of point of $\go$ in $\mathcal C_L^{\xi}$, and for that reason the two process 
should typically look the same.

\medskip

\noindent The result would not hold if $\hat \go$ had an intensity of a larger order.

\begin{proof}
Let $Q$ resp. $\tilde Q$ denote the law of $\go$ resp. $\tilde \go$ under $\bbP$.
For a function $f$, we denote by $Q(f)$ resp. $\tilde Q(f)$ expectation w.r.t $Q$ resp. $\tilde Q$. 
Note that $\tilde Q$ is absolutely continuous with respect to $Q$ and one has
\begin{equation}
\frac{\dd \tilde Q}{\dd Q}(\go)=\!\!\!\!\!\!  \prod_{\{(\go_i,r_i)\in \mathcal C'_L\times [\sqrt{d} L^{\xi}, 2\sqrt{d} L^{\xi}]\}} \!\!\! \!\!\! 
\left(1+(\alpha)^{-1} r_i^{1+\alpha}L^{-\frac{(d+1+\alpha)\xi+1}{2}}\right)
e^{-\frac{\sqrt{d}}{2}   L^{\frac{(d-1-\alpha)\xi+1}{2}}}.
\end{equation} 
For any event $A$ by Cauchy-Schwartz inequality, on has
\begin{equation}
Q(A)=\tilde  Q \left( \frac{\dd  Q}{\dd \tilde Q}\ind_A\right)\le \sqrt{ \tilde Q\left[\left( \frac{Q}{\tilde Q}\right)^2\right]}\sqrt{\tilde Q(A)}. 
\end{equation}
What is left to show is that the first term in the right-hand side remains bounded with $L$.
One has
\begin{multline}
Q\left( \frac{Q}{\tilde Q}\right):=\exp \left(\frac{\sqrt{d}}{2}   L^{\frac{(d-1-\alpha)\xi+1}{2}}\right)\\
\times Q\left(\prod_{\{(\go_i,r_i)\in \mathcal C'_L\times [\sqrt{d} L^{\xi}, 2\sqrt{d} L^{\xi}]\}} \frac{1}{
1+(\alpha)^{-1} r_i^{1+\alpha}L^{-\frac{(d+1+\alpha)\xi+1}{2}}}\right).
\end{multline}
And 
\begin{multline}
Q\left(\prod_{\{(\go_i,r_i)\in \mathcal C'_L\times [\sqrt{d} L^{\xi}, 2\sqrt{d} L^{\xi}]\}} \frac{1}{
1+(\alpha)^{-1} r_i^{1+\alpha}L^{-\frac{(d+1+\alpha)\xi+1}{2}}}\right)\\
=\exp\left(-\frac{L^{(d-1)\xi+1}}{2}\int_{\sqrt d  L^{\xi}}^{2\sqrt d  L^{\xi}}     \frac{L^{-\frac{(d+1+\alpha)\xi+1}{2}}\dd r}{
1+(\alpha)^{-1} r^{1+\alpha}L^{-\frac{(d+1+\alpha)\xi+1}{2}}}\right)  .
\end{multline}
Note that the quantity $r^{1+\alpha}L^{-\frac{(d+1+\alpha)\xi+1}{2}}$ is small uniformly in the domain of integration 
(by the assumption $\xi(d-1-\alpha)+1>0$) so that
\begin{multline}
 \int_{\sqrt d  L^{\xi}}^{2\sqrt d  L^{\xi}}     \frac{\dd r}{
1+(\alpha)^{-1} r^{1+\alpha}L^{-\frac{(d+1+\alpha)\xi+1}{2}}}\\
=\sqrt d  L^{\xi}-(1+o(1))L^{-\frac{(d+1+\alpha)\xi+1}{2}}\int_{\sqrt d  L^{\xi}}^{2\sqrt d  L^{\xi}}    \alpha^{-1}r^{1+\alpha} \dd r\\
=\sqrt d  L^{\xi}+O(L^{\frac{(\alpha+3-d)\xi-1}{2}})
\end{multline}
Putting everything together one gets
\begin{equation}
 Q\left( \frac{Q}{\tilde Q}\right)=\exp(O(1)).
\end{equation}

\end{proof}

\subsection{The effect of the change of measure}\label{tunapa}

In this section we estimate the 
difference between $\log Z_L^{\tilde \go}(\mathcal A_L^\xi)$ and  $\log Z_L^{\go}(\mathcal A_L^\xi)$.

\begin{proposition}\label{plez}
Suppose that 
\begin{equation}
(d-1-\alpha)\xi+1>0.
\end{equation}
Then, for any $\gep>0$, with probability tending to one when $N$ goes to infinity
\begin{equation}
\log Z_L^{\go}(\mathcal A_L^\xi) -\log Z_L^{\tilde \go}(\mathcal A_L^\xi) \ge L^{\frac{(d+1-\alpha-2\gamma)\xi+1}{2}-\gep}.
\end{equation}
\end{proposition}

The idea of the proof is quite simple (see figure \ref{yty2}). Making the change of environment $\go\to \tilde \go$, we add roughly $L^{\frac{(d-1-\alpha)\xi+1}{2}}$ traps of radius 
larger than $\sqrt{d}L^\xi$.
The traps we add are wide enough so that every trajectory in $\mathcal A_L^\xi$ has to go through every one of them (this explains our choice of adding only traps of large radius).

\begin{figure}[hlt]
\begin{center}
\leavevmode %\epsfysize =5 cm
\epsfxsize =14 cm
\psfragscanon
\psfrag{L2}{$L/2$}
\psfrag{L}{$L$}
\psfrag{R2}{$\R^{d-1}$}
\psfrag{R}{$\R$}
\psfrag{HL}{$\mathcal H_L$}
\psfrag{O}{$0$}
\epsfbox{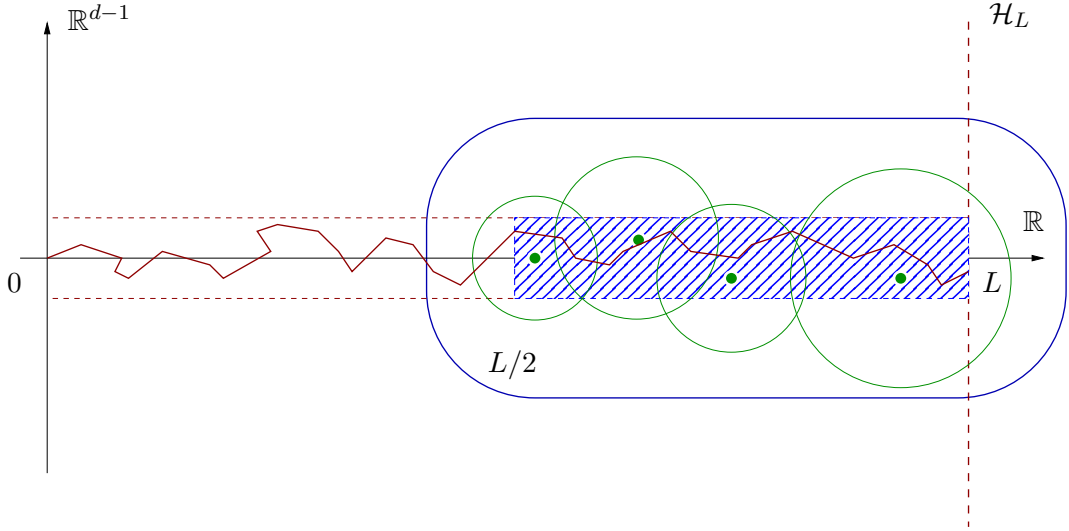}
\end{center}
\caption{\label{yty2}  The shadowed region is $\bar C_L^{\xi}$, and this is where $\go$ is modified.
The full line that encircles $\bar C_L^{\xi}$ denote the limit of $\tilde C_L^\xi$, the region where the values taken by $V^{\tilde \go}$ and $V^\go$ may differ.
A typical trajectory of $\mathcal A_L^{\xi}$  is represented. The four dots denote $\hat \go_i$, $i\in 1\dots 4$, the center of the traps that have been added.
The zone of influence of these traps $B(\hat \go_i, \hat r_i)$ are represented as circles.
The definition of $\hat \go$ implies that the radius of the traps are large enough to cover all the width of the tube $\mathcal C_L^\xi$ on a segment of length $L^\xi$.  
}
\end{figure}

\medskip

Under $\mu_L^{\go}$, trajectories are roughly ballistic, so that they should typically spend a time of order $L^\xi$ in each trap.
As the traps are of radius $\approx L^\xi$, they  modify the potential by  $L^{-\xi\gamma}$. Therefore, for most trajectories in $B\in\mathcal A_L^\xi$, one should have

\begin{multline}
\int_0^{T_{\mathcal H_L}} (V^{\tilde \go}-V^{\tilde \go})(B_t)\dd t = \int_0^{T_{\mathcal H_L}} (V^{\hat \go})(B_t)\dd t 
\\ \approx L^\xi \times L^{-\xi\gamma}\times \#\{\text{ traps in $\hat \go$ }\}\approx  L^{\frac{(d+1-\alpha-2\gamma)\xi+1)}{2}},
\end{multline}

which heuristically explains the result.

\medskip

To make this sketch rigorous, the main point is to give a proof of the fact that each trajectory spend a time of order $L^\xi$ in each trap.
This is the aim of Proposition \ref{tecnicos}.

For a given function $V: \R^d\to \R^+$ define the probability measure $\bar \mu^V$ by 
\begin{equation}
\frac{\dd \bar  \mu^V}{\dd \bP}(B):=  \frac{1}{Z^V_L(\mathcal A^{\xi}_L)}e^{-\int_{0}^{T_{\mathcal H_L}} (\gl+V(B_t))\dd t}\ind_{\mathcal A^{\xi}_L},
\end{equation}
where $Z^V_L(\mathcal A^{\xi}_L)$ is defined in the same way as $Z^{\go}_L(\mathcal A^{\xi}_L)$
with $V^\go$ replaced by $V$.

\medskip

Given $a$ in $[0,L]$ one wants to check that most trajectories of $\mathcal A^\xi_L$ spend a reasonable amount of time in
the slice of the tube
\begin{equation}
 [a,a+L^\xi]\times [-L^{\xi/2},L^{\xi}/2]^{d-1}.
\end{equation}
Let $B^{(1)}$ denote the first coordinate of $B$.

\begin{proposition}\label{tecnicos}
For any non-negative function $V$ such that $V(x)\le \log L$ for all $x$ such that $|x| \le L^2$, for any $\gep>0$
for $L$ large enough,  and for any $a\in [0,L-L^{\xi}]$,
\begin{equation}
\bar \mu^V \left(\int_{0}^{T_{\mathcal H_L}} \ind_{\{ B^{(1)}_t\in  [a,a+L^{\xi}]\}}\dd t  \le L^{\xi-\gep}\right)\le e^{-L^{\xi}}.
\end{equation}
\end{proposition}

\begin{rem}\rm
 The result above simply states that under the polymer measure, the cost for the motion to be superbalistic is roughly the same as for Brownian-Motion.
Even though the statement is quite natural, the proof we present contains some technicalities due to the fact the Brownian motion is 
confined in a tube and one has to control what happens close to the boundary. However the underlying idea is quite simple.
\end{rem}

We postpone the proof of this statement at the end of the section and use it to prove Proposition \ref{plez}.

Let $\mathcal N$ be the number of point in $\hat \go$ (this a Poisson variable of mean $(\sqrt{d}/2)L^{\frac{(d-1-\alpha)\xi+1}{2}}$).
We choose to index them in an arbitrary way so that one can write
\begin{equation}
\hat\go:=\{(\hat\go_kn,\hat r_k),\ k\in \{1,\dots,\mathcal N\}\}.
\end{equation}
For $i\in\{0,\dots, \mathcal N\}$,
define
\begin{equation}
\go^i:=\go \cup \{ (\hat \go_k,\hat r_k), k\in \{1,\dots,i\}\}
\end{equation}
Note that $\go^0=\go$ and $\go^{\mathcal N}:= \tilde\go$, so that $(\go^i)_{0\le i\le N}$ is an interpolating sequence between $\go$ and $\tilde \go$.

\begin{lemma} Given $\gep>0$,
 there exists a constant $c$ such that for all $L$ large enough, 
 for every environment $\tilde \go$ that satisfies 
 \begin{equation}
 \forall |x|\le L^2, V^{\tilde \go}(x)\le \log L,
 \end{equation}
  for all $i\in\{1,\dots, \mathcal N\}$
\begin{equation}\label{possey}
 Z^{\go^i}_{L}(\mathcal A^\xi_L)\le  Z^{\go^{i-1}}_{L}(\mathcal A^\xi_L) e^{-cL^{\xi(1-\gamma)-\gep}}.
\end{equation}
As a consequence 
\begin{equation}\label{poisse}
 Z^{\tilde\go}_{L}(\mathcal A^\xi_L)\le \exp(-c\mathcal N L^{\xi (1-\gamma)-\gep})Z^{\go}_{L}(\mathcal A^\xi_L).
\end{equation}
\end{lemma}

\begin{proof}

As $\go_i$ lies in $\bar C_L^\xi$ and $\hat r_i\ge \sqrt{d}L^{\xi}$, there exists $a_i\in(L/2,L-L^\xi)$ such that
\begin{equation}
 \left[a_i,a_i+L^{\xi}\right]\times[-L^{\xi}/2,L^\xi/2]^{d-1}\subset \mathcal B(\hat \go_i,
\hat{r_i}).
\end{equation}

By assumption, for $|x|\le L^2$ and for all $i\le \mathcal N$,   $V^{\go^i}(x)\le V^{\tilde\go}(x)\le \log L$.
Therefore one can apply Proposition \ref{tecnicos}. Using  $\hat r_i\le 2\sqrt{d}L^{\xi}$ one gets
 
\begin{multline}
\frac{Z^{\go^i}_{L,\gb}(\mathcal A^\xi_L)}{Z^{\go^{i-1}}_{L,\gb}(\mathcal A^\xi_L)}=
\bar \mu^{V^{\go^{i-1}}} \left(e^{-\int_{0}^{T_{\mathcal H_L}} \hat r_i^{-\gamma} \ind_{\{B_t\in\mathcal B(\hat \go_i,
\hat r_i) \}}\dd t} \right)\\
\le \bar \mu^{V^{\go^{i-1}}} \left(e^{-\int_{0}^{T_{\mathcal H_L}}  (2\sqrt{d}L^{\xi})^{-\gamma} \ind_{\{B_t^{(1)}\in  \left[a_i,a_i+L^{\xi}\right] \}}\dd t}\right)\\
\le  \bar \mu^{V^{\go^{i-1}}} \left(\int_{0}^{T_{\mathcal H_L}} \ind_{\{B_t^{(1)}\in  \left[a_i,a_i+L^{\xi}\right]\dd t \le L^{\xi-\gep} \}}\right)+e^{-L^{\xi-\gep}(2\sqrt{d}L^{\xi})^{-\gamma}}\\
\le  e^{-L^{\xi}}+e^{-L^{\xi-\gep}(2\sqrt{d}L^{\xi})^{-\gamma}}                 \le e^{-cL^{\xi(1-\gamma)-\gep}}.
\end{multline}

\end{proof}

Proposition \ref{plez} is immediate consequence of \eqref{poisse} and the fact that $\mathcal N$ is a Poisson variable of 
mean $(\sqrt{d}/2) L^{\frac{(d-1-\alpha)\xi+1}{2}}$.

\begin{proof}[Proof of Proposition \ref{tecnicos}]

Define
\begin{equation}\begin{split}
T_1&:=\inf\{t\ge 0 \ : \  B^{(1)}_t=a +L^{\xi}/2\}, \\
 T_2&:=\inf\{t\ge T_1 \ : \ |B^{(1)}_t-(a+L^{\xi}/2)|=L^\xi/2\}.
\end{split}\end{equation}
For all $t\in [T_1,T_2]$, $B^{(1)}_t\in [a,a+L^{\xi}]$ and thus it is sufficient to prove that with large $\bar \mu^V$ 
probability $T_2-T_1$ is large. 
Set
\begin{equation}
U_1:=B_{T_1} \quad \text{ and } \quad  U_2:=B_{T_2}. 
\end{equation}
Note that conditionally on $U_1=x$, $U_2=y$, the law of $(B_t)_{t\in[T_1,T_2]}$ under $\bar\mu^V$
is independent of the rest of the motion, and that (recall that $\bP_x$ is the law of a standard motion started from $x$)

\begin{multline}\label{ReQ}
 \bar \mu^V \left(T_2-T_1
\le L^{\xi-\gep} \ | \ U_1=x ;\  U_2=y \right)\\
= \frac{
\bE_x\left[e^{-\int_{0}^{T_{2}} (\gl+V(B_t))\dd t}\ind_{\{\forall t\in(0,T_2), B_t\in \mathcal C_L\}}\ | \ B_{T_2}=y\right]}
{\bE_x\left[e^{-\int_{0}^{T_{2}} (\gl+V(B_t))\dd t}\ind_{\{\forall t\in(0,T_2), B_t\in \mathcal C_L\}}\ | \ B_{T_2}=y\right]}
=:\frac{R(x,y)}{Q(x,y)}.
\end{multline}
We are to show that this is small uniformly in the choice of $x$ and $y$.
Our way to estimate both term in the fraction is  to suppress inhomogeneity due to the potential $V$: using the assumption
$0\le V\le \log N$ we can replace $V(B_t)$ by $0$ in $R$ and by $\log L$ in $Q$.

Let $\bar x$ resp.\ $\bar y$ be 
the projection of the $d-1$ last coordinate of $x$ resp.\ y on $\bbR^{d-1}$, and 
let $p^*_t(\cdot,\cdot)$ denote the heat kernel on $[-L^{\xi/2},L^{\xi/2}]^{d-1}$ with Dirichlet boundary condition.
One has

\begin{equation}\begin{split}
R(x,y)&\le \int_{(0,L^{\xi-\gep})}p_t^*(\bar x,\bar y)e^{-\gl t}\bP_x(T_2\in \dd t).\\ 
Q(x,y)&\ge \int_{(0,\infty)}p_t^*(\bar x,\bar y)e^{-(\gl+\log L) t} \bP_x(T_2\in \dd t).
\end{split}\end{equation}

 Note that uniformly on $t\le L^{2\xi-\gep}$, when $L$ gets large
\begin{multline}
\bP_x(T_2\le t)=
\bP(\max_{s\in [0,1]} |B_s^{(1)}|\ge L^{\xi}/(2\sqrt{t}))=2(1+o(1))\bP(\max_{s\in [0,1]} B_s^{(1)}\ge L^{\xi}/(2\sqrt{t}))\\
=\frac{4(1+o(1))}{\sqrt{2\pi}}\int_{L^{\xi}/(2\sqrt{t})}^\infty e^{-\frac{u^2}{2}}\dd u
=\frac{8\sqrt{t}(1+o(1))}{L^{\xi}\sqrt{2\pi}}
e^{-\frac{L^{2\xi}}{8t}}.
\end{multline}

Using the estimates on  heat-kernel  from the Appendix (Lemma \ref{diric})
one obtains by using integration by part that for large $L$
\begin{multline}\label{air}
R(x,y)\le A(\bar x,\bar y)\int_{(0,L^{\xi-\gep})}\frac{1}{t^2} \bP_x(T_2\in \dd t)\\
=(1+o(1)) A(\bar x,\bar y) 
\int_{0}^{L^{\xi-\gep}} \frac{16}{t^{5/2}\sqrt{2\pi}L^{\xi}}e^{-\frac{L^{2\xi}}{8t}} \dd t
\le A(\bar x,\bar y) e^{-L^{\xi+\gep}/16}
\end{multline}
and that
\begin{multline}\label{ikik}
Q(x,y)\ge A(\bar x,\bar y)\int_{0}^{\infty} e^{-(\gl+\log L)t}  e^{-L^{2\xi}{t}-\frac{\pi^2}{2 t}} \bP_x(T_2\in \dd t)
\\ \ge  A(\bar x,\bar y) \int_{L^{\xi}/2}^{L^{\xi}} e^{-(\gl+\log L)t}  e^{-\frac{L^{2\xi}}{t}-\frac{\pi^2t}{2}} \bP_x(T_2\in \dd t)\dd t
\\ \ge A(\bar x,\bar y)e^{-2 L^{\xi}(\log L)}.
 \end{multline}

 where
 \begin{equation}
 A(u,v):=  \prod_{i=1}^{d-1} \min\left( (u_i+L^{\xi/2}), (L^{\xi/2}-u_i)\right)
 \min\left( (v_i+L^{\xi/2}), (L^{\xi/2}-v_i)\right).
 \end{equation}
% (We do not take the integral over $(0,\infty)$ in \eqref{ikik} because we want  Lemma \ref{excu} to remain valid).
% Integrating \eqref{air} and \eqref{ikik} one get that there exists a $c$ such that for all $L$ large enough
% \begin{equation}
% Z_2^{x,y}((T_2-T_1)\le L^{\xi-\gep} )\le A(x,y)e^{-c L^{\xi+\gep}},
% \end{equation}
% and
% \begin{equation}
%   Z_2^{x,y}\ge A(x,y)  e^{-\frac{1}{c}2L^{\xi}\sqrt{\log L}}.
% \end{equation}
which (recall \eqref{ReQ}) gives the result.

\end{proof}

\subsection{Proof of Theorem \ref{superdiff}}

Now we can use the results of all the previous sections to get the main Theorem.
Consider $\xi<\bar\xi(d,\alpha,\gamma)$. One can check that it satisfies both
\begin{equation}\label{condit}\begin{split}
\frac{(d+1-\alpha-2\gamma)\xi+1}{2}&> 2\xi-1.\\
(d-1-\alpha)\xi+1&>0.
\end{split}
\end{equation}
Then with probability going to one (cf. Proposition \ref{preswe}), one has

\begin{equation}
\log Z_L^{\go}(\mathcal A_L^{\xi})\le \log Z_L^{\go}(\cB_L^{\xi})+L^{2\xi-1}(\log L)^3.
\end{equation}
Then combining this with Proposition \ref{plez} and  $Z_L^{\go}(\cB_L^{\xi})=Z_L^{\tilde \go}(\cB_L^{\xi})$ one get that with probability tending to one

\begin{equation}
\log Z_L^{\tilde \go}(\mathcal A_L^{\xi})\le \log Z_L^{\tilde \go}(\cB_L^{\xi})+L^{2\xi-1}(\log L)^3-L^{\frac{(d+1-\alpha-2\gamma)\xi+1}{2}-\gep}.
\end{equation}
Hence using \eqref{condit} and choosing $\gep$ small enough, one gets that with probability tending to one when $L\to \infty$ 
\begin{equation}
\mu_L^{\tilde \go}(\mathcal A_L^{\xi})\le \exp(-L^{\frac{(d+1-\alpha-2\gamma)\xi+1)}{2}-\gep}/2).
\end{equation}

Using Lemma \ref{mezchange} for the event $A:=\{ \go \ | \  \mu_L^{\go}(\mathcal A_L^{\xi})> \exp(-L^{\frac{(d+1-\alpha-2\gamma)\xi+1)}{2}-\gep}/2)\}$
one gets that
\begin{equation}
\lim_{L\to \infty} \bbP\left[ \mu_L^{\go}(\mathcal A_L^{\xi})> \exp(-L^{\frac{(d+1-\alpha-2\gamma)\xi+1}{2}}-\gep/2)\right]=0.
\end{equation}
Which ends the proof.

 \qed

{\bf Acknowledgements:} The author would like to thank to A.S.\ Sznitman 
for bringing to his attention the work of M. W\"uhtrich \cite{cf:W1,cf:W2,cf:W3} 
which was a great source of inspiration for the present paper as well as the anonymous referee, for his suggestion to 
improve and shorten
the proof of Proposition \ref{tecnicos}. 
This work was initiated during the authors stay in Universit\`a di Roma Tre 
under the support of European Research Council through the advanced grant PTRELSS 228032, 
he acknowledges kind hospitality and support.

\appendix

\section{Estimates}

\begin{lemma}\label{troto}
For all $L$ large enough
\begin{equation}
\bbP\left[ \max_{x\in[-L^2,L^2]^d} V^{\go}(x)> \log L\right] \le 1/L. 
\end{equation}
\end{lemma}

 \begin{proof}
 
By translation invariance, it is sufficient to show that
\begin{equation}
 \bbP\left[ \max_{x\in [-1,1]^d} V^{\go}(x)\ge \log L \right]\le \frac{1}{L^{2d+1}}.
\end{equation}
 Note that 
\begin{equation}
 \max_{x\in [-1,1]^d} V^{\go}(x)\le  \max_{x\in B(0,\sqrt{d})} V^{\go}(x)\le \sum_{i=0}^N r_i^{-\gamma}\ind_{|\go_i|\le r_i+\sqrt{d}}:= X_1.
\end{equation}
Using standard properties of Poisson Point Processes one gets that
\begin{equation}
 \bbE\left[\exp(aX_1)\right]= \exp\left(\int_{1}^{\infty} \alpha^{-1}\sigma_d (r+\sqrt{d})^{d-\alpha-1}(e^{a r^{-\gamma}}-1) \dd r\right)<\infty, 
\end{equation}
where $\sigma_d$ is the volume of the unit d-dimensional euclidian ball.
Therefore one has
\begin{equation}
 \bbP\left[ \max_{x\in B(0,1)} V^{\go}(x)\ge \log L \right]\le \bbE\left[e^{(2d+2)X_1-(2d+2)\log L}\right].
\end{equation}
and the right-hand side is less than $L^{-2d+1}$ for $L$ large enough.
 \end{proof}

Let $p^*_t$ be the heat kernel on $[-L^{\xi}/2,L^{\xi}/2]^{d-1}$ with Dirichlet boundary condition
For $x$ in $[-L^{\xi}/2,L^{\xi}/2]$ set 
\begin{equation}
 A(x):=\prod_{i=1}^d\min\left( (x_i+L^{\xi}/2), (L^\xi/2-x_i)\right)
\end{equation}

\begin{lemma}\label{diric}
 One has that for every $t\ge 0$
\begin{equation}
p^*_t(x,y)\le A(x)A(y)/t^2
\end{equation}
and for all $t$ large enough (where large enough does not depend on $L$).
\begin{equation}
 p^*_t(x,y)\ge A(x)A(y)e^{-\frac{L^{2\xi}}{t}-\frac{t\pi^2}{2}}
\end{equation}
\end{lemma}

\begin{proof}
First we remark that due to the product structure of the Kernel, it is sufficient to treat the one dimensional case $(d-1)=1$.
One considers first  $p_t^{0,*}$ the heat kernel on $[0,1]$ with Dirichlet boundary condition.
Diffusive scaling gives
\begin{equation}
 p^*_t(x,y)=L^{-\xi} p^*_{L^{-2\xi}t}(L^{-\xi}x+1/2,L^{-\xi}y+1/2).
\end{equation}

 A decomposition of the Dirac distribution $\gd_x$ onto the base of eigenfunction of $\gD/2$ with Dirichlet boundary condition gives 
\begin{multline}
 p_t^{0,*}(x,y)=\sum_{k=1}^{\infty} 2\sin (k\pi x) \sin (k\pi y) e^{-(k\pi)^2 t}\\
\le \min(x,1-x)\min(y,1-y)\sum_{k=1}^{\infty} 2\pi^2 k^2 e^{-\frac{(k\pi)^2t}{2}}\\
%\le  \min(x,1-x)\min(y,1-y)\sum_{k=1}^{\infty}
% 2\pi^2 k e^{-\frac{k\pi^2 t}{2}}\\
% =2\min(x,1-x)\min(y,1-y)\sinh(\pi^2 t/4)^{-2}\\
\le \frac{8}{\pi^2 t^2} \min(x,1-x)\min(y,1-y).
\end{multline}
which once rescaled, gives the desired upper-bound.
Now we perform a lower bound on $p_t^{0,*}$ for large $t$, indeed using similar computation one gets that there exist a constant $C$ such that of all $t\ge 1$
\begin{equation}
 |\sum_{k=2}^{\infty} 2\sin (k\pi x) \sin (k\pi y) e^{-\frac{(k\pi)^2 t}{2}}|
\le C \min(x,1-x)\min(y,1-y) e^{-2\pi^2 t}.
\end{equation}
Hence for $t$ large enough.
\begin{multline}\label{pstar0}
 p_t^{0,*}(x,y)\ge 2\sin (\pi x) \sin (\pi y) e^{-\frac{\pi^2 t}{2}}-|\sum_{k=2}^{\infty} 2\sin (k\pi x) \sin (k\pi y) e^{-\frac{(k\pi)^2t}{2} }|\\
\ge  
\min(x,1-x)\min(y,1-y)e^{-\frac{\pi^2 t}{2}}.
\end{multline}
However this is not sufficient to get directly by rescaling the lower bound for $ p^*_t(x,y)$ for $t\le L^{2\xi}$.

To do so define $x^*$ (and $y^*$ is defined similarly) as
\begin{equation} \begin{split}
 x^*&:= x \quad \text{ if } x\in[-(L^{\xi}-1)/2/2,(L^{\xi}-1)/2],\\
 x^*&:= (L^{\xi}-1)/2 \quad \text{ if } x\ge (L^{\xi}-1)/2,\\
 x^*&:= -(L^{\xi}-1)/2 \quad \text{ if } x\le -(L^{\xi}-1)/2,
\end{split}\end{equation}

One uses the following comparison argument 
\begin{multline}
 p_t^{*}(x,y)\dd y=\bP_x\left[B_s\in [-L^{\xi}/2,L^{\xi/2}];\ \forall s\in[0,t], B_t\in \dd y\right]\\
\ge \bP_x\left[ |B_s-sy^*+(1-s)x^*|\le 1/2;\ \forall s\in [0,1], B_t\in \dd y\right]\\
 = e^{\frac{1}{t}\left(-(y-x)(x^*-y^*)+\frac{(x^*-y^*)^2}{2}\right)}  p^{0,*}_t(x-x^*+1/2,y-y*+1/2)\dd y,
\end{multline}
where last inequality is just Girsanov Path Transform.
Then we use \eqref{pstar0} to get the result.

\end{proof}

\section{The point to point model}

\subsection{Result}
In this Section we show that a weakened version of our result holds for the so-called point-to-point model.
Consider $\go$ and $V^\go$ defined as for the other model.
Given $y\in \bbR^d$ we define $B(y):=B(y,1)$ to be the euclidian ball of radius one centered on $y$ and $T_y$ to be the first hitting time of $  B(y)$.
For $L>0$ set $y_L:=(L,0,\dots,0)$ and

\begin{equation}
Z^\go_L:= \bE \left[e^{-\int_0^{T_{y_L}} (\gl+V^\go(B_t)) \dd t }\ind_{\{T_{y_L}<\infty\}}\right].
\end{equation}
The associated path measure $\mu_{L}^{\go}$ is given by

\begin{equation}
\frac{\dd \mu_L^{\go}}{\dd P}:=\frac{1}{Z^\go_L}e^{-\int_0^{T_{y_L}} (\gl+V(B_t)) \dd t }\ind_{T_{y_L}<\infty}.
\end{equation}
Consider $\mathcal C_L^{\xi}$ as in \eqref{clxi} and define

\begin{equation}
\mathcal A_L^{\xi}:= \{ T_{y_L}<\infty, \forall  t \in [0, T_y], B_t \in \mathcal C_L^{\xi}\}.
\end{equation}

One has the following weakened version of Theorem \ref{superdiff}

\begin{theorem}
For any
\begin{equation}
\xi<\frac{1}{1+\alpha+2\gamma-d}
\end{equation}
one has
\begin{equation}
\lim_{L\to\infty}\bbE\left[ \mu_L^{\go}(\mathcal A_L^\xi)\right]=0.
\end{equation}
\end{theorem}

\subsection{Sketch of proof}
One defines (in analogy with \eqref{bariri}, \eqref{tildidi})
\begin{equation}\label{barccc}
\bar C_L^\xi := [L/4,3L/4] \times [-L^{\xi}/2,L^{\xi}/2]^{d-1}
\end{equation}
and 
\begin{equation}
\tilde C_L^\xi:= \{ x\in \bbR^d \ | \ d(x, \bar C_L)>2 \sqrt{d} L^{\xi} \}=  \bigcup_{y\in \bar C_L} B(y,2 \sqrt{d} L^{\xi}),
\end{equation}

Let $\mathcal B^{\xi}_L$ be the set of trajectories that avoids the set $\tilde C_L^\xi$.
\begin{equation}
\mathcal B^{\xi}_L:=  \{ B \ | \ \forall t \in [0, T_{y_L}], B_t\notin \tilde C_L^\xi \}.
\end{equation}

The strategy we use here, is to get first a weak comparison between $Z_L^\go (\mathcal B^{\xi}_L)$ and $Z_L^\go (\mathcal A^{\xi}_L)$ similar to the one of Proposition
\ref{preswe}, 
and then to upgrade it by modifying the environment, adding traps of radius $(\sqrt{d}L^\xi,2\sqrt{d}L^\xi)$ in $\bar C_L^\xi$.

Fix $\xi<\frac{1}{1+\alpha+2\gamma-d}$.
We assume that is bounded by $\log L$ on $B(0,L^2)$
(that happens with probability larger than $1-1/L$ according to Lemma \ref{troto}).

\medskip
Let $e_1=(1,0, \dots,0)$ and $e_2=(0,1,0 \dots,0)$ be the two first coordinate vector in $\bbR^d$'s canonical base.
For $i\in \bbN$ one defines
\begin{equation}
  \mathcal C_{L,i}^{\xi}+\mathcal C_L^{\xi}+i(1+2\sqrt{d})L^\xi e_2.
 \end{equation}
 And set 
 \begin{equation}\begin{split}
 x_{i,L}&:=i(1+2\sqrt{d})L^\xi e_2,\\
 y_{i,L}&:=Le_1 +i(1+2\sqrt{d})L^\xi e_2,\\
 \mathcal D_{i,L}^{(1)}&:= \bigcup_{\alpha\in [0,1]} B(\alpha x_{i,L},2),\\
 \mathcal D_{i,L}^{(2)}&:=  \mathcal D_{i,L}^{(1)}+Le_1.
 \end{split}\end{equation}
 The set  $\mathcal D_i^{(1)}$ resp. $\mathcal D_i^{(2)}$ is the set of points whose distance to the segment $[0, x_{i,L}]$ resp. $[y_L, y_{i,L}]$ less than two. 
 
  Finally $\rho$ denote the translation of vector $(1+2\sqrt{d})L^\xi e_2$.
  
  \medskip

  Consider the family of events 
\begin{multline}
\mathcal A_{L,i}^{\xi}:=\{T_{x_{i,L}}<T_{y_{i,L}}<T_{y_L}<\infty, \forall t\in (T_{x_{i,L}},T_{y_{i,L}}), B_t \in  \mathcal C_L^{\xi},\\
 \forall t<T_{x_{i,L}}, B_t\in  \mathcal D_{i,L}^{(1)},  \forall t\in (T_{y_{i,L}}, T_{y_{L}}), B_t\in  \mathcal D_{i,L}^{(2)},\}.
\end{multline}

Note that  these events are disjoint (see figure \ref{yty3}), and that for all $i\ne 0$, $\mathcal A_{L,i}^{\xi}\subset \mathcal B_L^\xi$, and hence
\begin{equation}
Z^\go_L(\mathcal B_L^\xi) \ge \sum_{i\in\{-\log L,\dots \log L\}\setminus \{0\}} Z^\go_L(A_{L,i}^\xi).
\end{equation}

\begin{figure}[hlt]
\begin{center}
\leavevmode %\epsfysize =5 cm
\epsfxsize =14 cm
\psfragscanon
\psfrag{L/4}{$L/4$}
\psfrag{3L/4}{$3L/4$}
\psfrag{L}{$L$}
\psfrag{Rd}{$\R^{d-1}$}
\psfrag{R}{$\R$}
\psfrag{CLXI}{$\tilde C_L^\xi$}
\psfrag{O}{$0$}
\psfrag{D1}{$\mathcal D^{(1)}_{1,L}$}
\psfrag{D2}{$\mathcal D^{(2)}_{1,L}$}
\psfrag{yl1}{$y_{1,L}$}
\psfrag{yl}{$y_{L}$}
\psfrag{xl1}{$x_{1,L}$}
\epsfbox{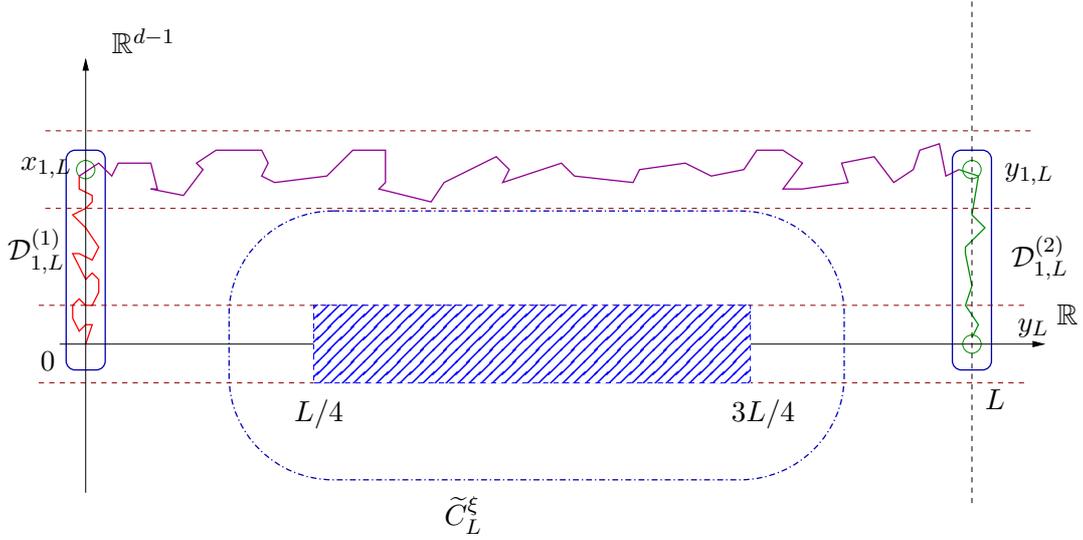}
\end{center}
\caption{\label{yty3}  Here is a two-dimensional projection of a trajectory in $\mathcal A_{L,1}^\xi$. First it reaches $B(x_{L,1})$ while staying in the 
narrow tube $\mathcal D^{(1)}$ then it goes from $B(x_{L,1})$ to  $B(y_{L,1})$ and stays in the tube $\mathcal C^\xi_{L,1}$ (dashed line), and then 
stays in $\mathcal D^{(1)}$ till it hits $B(y)$. It appears clearly on the figure that such trajectories cannot hit $\tilde C_L^\xi$ whose limits are represented 
by the thick dashed line (the shadowed region is  $C_L^\xi$). The three parts of the path corresponding to three terms in  the decomposition  \eqref{dsk} are draw in different colors}
\end{figure}
 Then by using the Markov property at time $T_{x_{i,L}}$ and $T_{y_{i,L}}$, one gets that 
 
\begin{multline}\label{dsk}
 Z^\go_L(A_{L,i}^\xi)
\ge 
\bE\left[ e^{-\int_0^{T_{x_{L,i}}} (\gl+V^\go(B_t)) \dd t } \ind_{\{T_{x_{i,L}}<\infty;  \forall t<T_{x_{i,L}}, B_t\in  \mathcal D_{i,L}^{(1)}\}}\right]\\
\times 
\inf_{u\in B(x_{L_i})}\bE_{u}\left[e^{-\int_0^{T_{y_{L,i}}} (\gl+V^\go(B_t)) \dd t }\ind_{\{T_{y_{i,L}}<\infty; \forall t\le T_{y_{i,L}}, B_t \in \mathcal C_{L,i}^{\xi}\}} \right]\\
\times \inf_{v \in B(y_{L_i})}\bE_{v}\left[e^{-\int_0^{T_{y_{L}}} (\gl+V^\go(B_t)) \dd t }\ind_{\{T_{y_{L}}<\infty; \forall t\le T_{y_{L}},  B_t\in  \mathcal D_{i,L}^{(2)}\}} \right].
\end{multline}

Then one remarks that the first and third term can be bounded by using standard tubular estimate for Brownian Motion (see for instance (1.11) of \cite{cf:SCPAM}).
Indeed in $ D_{i,L}^{(j)}$ under our assumption the potential is less than $\log L$, there exists  $C$ such that both terms are larger than $\exp(- Ci L^\xi\log L)$ for  and $L$ large enough.

As for the second term, using Proposition 2.2 in \cite[Chapter 5]{cf:S}, one get that it is larger than
\begin{equation}
e^{-C\log L} \bE_{x_{L_i}}\left[e^{-\int_0^{T_{y_{L,i}}} (\gl+V^\go(B_t)) \dd t }\ind_{\{T_{y_{i,L}}<\infty; \forall t\le T_{y_{i,L}}, B_t \in \mathcal C_{L,i}^{\xi}\}} \right]
=e^{-C\log L}  Z_L^{\rho^{-i} \go}(\mathcal A_L^\xi)
\end{equation}

We write the conclusion of this as

\begin{lemma}
With probability larger than $1-1/L$. 
\begin{equation}
\log Z^\go_L(\mathcal B^\xi_L)\ge  \max_{i\in \{-\log L,\dots \log L\}\setminus \{0\} } \log Z_L^{\rho^{-i} \go}(\mathcal A_L^\xi)-C' (\log L^2)L^\xi.
\end{equation}
\end{lemma}

The above Lemma plays the role of Lemma \ref{milor} for the point to plane model.
The reason why the result we obtain at the end is not as good as for the point to plane model is that the $L^{2\xi-1}$ of Lemma \ref{milor}
is replaced by $L^\xi$ here.

One can also get an equivalent of Lemma \ref{rotatruc} (the proof being exactly analogous)

\begin{lemma}
\begin{equation}
\bbP\left[\log Z_L^{\go}(\mathcal A_L^\xi)> \max_{i\in \{-\log L,\dots, \log L\}} \log Z_L^{\rho^{-i} \go}(\mathcal A_L^\xi) \right]\le \frac{1}{\log L}
\end{equation}
\end{lemma}

and thus of Proposition \ref{preswe}

\begin{proposition} \label{arabiata}
\begin{equation}
\bbP\left[\log Z_L^{\go}(\mathcal A_L^\xi)> \log Z^\go_L(\mathcal B^\xi_L)+C' (\log L^2)L^\xi \right]\le \frac{2}{\log L}.
\end{equation}
\end{proposition}

The rest  of the proof being exactly similar to the point-to-plane case we survey it very briefly.
We consider $\tilde \go$ constructed just as in Section \ref{addt} but with the definiton of $\bar C^\xi$ replaced by \eqref{barccc}
(here is is important to notice that our choice for $\xi$ implies $(d-1-\alpha)\xi+1>0$).
Obviously  Lemma \ref{mezchange} is still valid with this modifications. Then proof of Proposition \ref{plez} can be adapted without  difficulties and one gets that, for every $\gep>0$,
with probability tending to one, 
\begin{equation}
 \log Z_L^{\tilde \go}(\mathcal A_L^\xi)\le \log Z_L^{\go}(\mathcal A_L^\xi)-L^{\frac{(d+1-\alpha-2\gamma)\xi+1}{2}-\gep}\\
\end{equation}
while $ \log Z_L^{\tilde \go}(\mathcal B_L^\xi)=\log Z_L^{\go}(\mathcal B_L^\xi)$.
This together with Proposition \ref{arabiata} implies
\begin{equation}
\log Z_L^{\tilde \go}(\mathcal B_L^\xi)\ge \log Z_L^{\tilde \go}(\mathcal A_L^\xi)-C' (\log L^2)L^\xi +L^{\frac{(d+1-\alpha-2\gamma)\xi+1}{2}-\gep},
\end{equation}
With our choice of $\xi$, for $\gep$ sufficiently small, $L^{\frac{(d+1-\alpha-2\gamma)\xi+1}{2}-\gep}-C' (\log L^2)L^\xi$ tends to infinity,
and that ends the proof.

\end{document}